\RequirePackage{rotating}
\documentclass[12pt]{amsart}
\usepackage[graphicx]{realboxes}
\usepackage{float}
\restylefloat{table}
\usepackage{placeins}

\usepackage{amsmath}
\usepackage{amssymb}
\usepackage{epsfig}
\usepackage{wasysym}
\usepackage{graphicx}
\usepackage{tikz}
\usepackage{tikz-cd}
\usepackage{bm}
\usepackage{xcolor}
\usepackage{listings}
\numberwithin{equation}{section}
\usepackage{extpfeil}
\usepackage{array}
\usepackage{adjustbox}
\usepackage{longtable}
\usepackage{makecell}
\usepackage[all]{xy}
\usepackage{longtable}
\usepackage{rotating}
\input xy
\xyoption{all}
\usepackage{multirow}

\usetikzlibrary{positioning}
\usepackage[colorlinks=false,urlbordercolor=white]{hyperref}
  
\tikzset{sgplattice/.style={inner sep=1pt,norm/.style={red!50!blue},char/.style={blue!50!black},
  lin/.style={black!50}},cnj/.style={black!50,yshift=-2.5pt,left=-1pt of #1,scale=0.5,fill=white}}

\DeclareFontFamily{U}{mathb}{\hyphenchar\font45}
\DeclareFontShape{U}{mathb}{m}{n}{
      <5> <6> <7> <8> <9> <10> gen * mathb
      <10.95> mathb10 <12> <14.4> <17.28> <20.74> <24.88> mathb12
      }{}
\DeclareSymbolFont{mathb}{U}{mathb}{m}{n}
\DeclareMathSymbol{\righttoleftarrow}{3}{mathb}{"FD}

\calclayout
\allowdisplaybreaks[3]

\theoremstyle{plain}
\newtheorem{prop}{Proposition}[section]

\newtheorem{lemm}[prop]{Lemma}

\theoremstyle{definition}

\newtheorem{rema}[prop]{Remark}

\newcommand\symb[1]{\langle #1 \rangle}

\def\ra{\rightarrow}

\def\cH{{\mathcal H}}

\def\cK{{\mathcal K}}

\def\cM{{\mathcal M}}

\def\cS{{\mathcal S}}

\def\fS{{\mathfrak S}}

\def\mh{{\mathfrak h}}
\def\fS{{\mathfrak S}}

\def\bP{{\mathbb P}}
\def\bQ{{\mathbb Q}}

\def\bZ{{\mathbb Z}}

\def\bM{{\mathbb M}}

\def\bC{{\mathbb C}}

\def\C{C_N\times C_{MN}}

\def\Tor{\mathrm{Tors}}

\def\SL{\mathrm{SL}}

\def\Hom{\mathrm{Hom}}

\def\lim{\mathrm{lim}}

\makeatother
\makeatletter

\begin{document}

\title[Modular symbols]{Modular Symbols and Equivariant Birational Invariants}

\author[Zhijia Zhang]{Zhijia Zhang}

\address{
Courant Institute,
  251 Mercer Street,
  New York, NY 10012, USA
}

\email{zhijia.zhang@cims.nyu.edu}

\date{\today}

\begin{abstract}
We study relations between the classical modular symbols associated with congruence subgroups and Kontsevich-Pestun-Tschinkel groups $\cM_n(G)$ associated with finite abelian groups $G$.
\end{abstract}

\maketitle

\section{Introduction}
\label{sect:intro}

Let $G$ be a finite abelian group, acting regularly and generically freely on a smooth projective variety of dimension $n\geq 2$ over an algebraically closed field of characteristic zero. An {\em equivariant birational invariant} of such actions was introduced in \cite{KPT}. It takes values in the abelian group 
$$
\cM_n(G),
$$
defined via explicit generators and relations. This group and its generalizations in \cite{BnG} encode intricate geometric information,  leading to new results in equivariant birational geometry, see, e.g., \cite{HKTsmall}, \cite{KT-vector}, \cite{TYZ} and \cite{TYZ-3}. On the other hand, the simplicity of the defining relations of $\cM_n(G)$ reveals a rich arithmetic nature: it was found in \cite{KPT} that $\cM_n(G)$ carry Hecke operators, formal (co-)multiplication maps, and are closely related to Manin's modular symbols for modular forms of weight $2$, when $n=2$.


In this note, we continue the investigation of arithmetic properties of $\cM_n(G)$, with a particular focus on their relations with Manin symbols. Our main results are:
\begin{itemize}
    \item We settle the algebraic structure of $\cM^-_2(G)$, a quotient of the group $\cM_2(G)$, for any finite abelian group $G$, see Proposition \ref{prop:mainformminus}. The key ingredient is the construction of an isomorphism between
 $\cM_2^-(G)$ and the $\bZ$-module of classical Manin symbols for certain congruence subgroups.
 \item 
 We prove a conjecture from \cite[Section 11]{KPT} regarding the $\bQ$-ranks of $\cM_2(G)\otimes\bQ$ when $G$ is cyclic, and generalize the result to any finite abelian group $G$.
    \end{itemize}
    
Here is the roadmap of the paper. In Section~\ref{sect:back}, we recall relevant definitions. In Section ~\ref{sect:modsymb}, we study the connections between Manin symbols and the groups $\cM^-_2(G)$. Dimensional formulae for $\cM_2(G)\otimes\bQ$ are given in Section~\ref{sect:formulae}.

\

\noindent
{\bf Acknowledgments:} The author is grateful to Yuri Tschinkel and Brendan Hassett for many helpful conversations.

\section{Background}\label{sect:back}
Let $G$ be a finite {\em abelian} group, 
$
G^\vee=\Hom(G,\bC^\times)
$ 
its character group, $n$ a positive integer and
$$
\cS_n(G) 
$$
the $\bZ$-module freely generated by $n$-tuples of characters of $G$:
$$
\beta=(b_1,\ldots, b_n),\quad \text{such that}\quad \sum_{j=1}^n\bZ b_j=G^\vee.
$$

The group $\cM_n(G)$ is defined via the quotient 
$$
\cS_n(G) \to \cM_n(G)
$$
by the {\em reordering relation} 

\

\noindent
{\bf (O)}: for all $\beta=(b_1,\ldots, b_n)$ and all $\sigma\in \fS_n$, one has
$$
\beta= \beta^{\sigma}:=(b_{\sigma(1)},\ldots, b_{\sigma(n)}),
$$
\noindent and the {\em motivic blowup relation} 

\

\noindent
{\bf (M)}: for $\beta=(b_1,b_2,b_3,\ldots, b_n)$, one has
$
\beta=\beta_1 +\beta_2,
$
where 
\begin{equation*}
\beta_1:= (b_1-b_2, b_2, b_3, \ldots, b_n), \quad 
\beta_2:= (b_1, b_2-b_1, b_3,\ldots, b_n), \quad n\ge 2.
\end{equation*}
A closely related group $\cM^-_n(G)$  is defined as the quotient of $\cS_n(G)$ by {\bf (O)}, {\bf (M)}  and the {\em anti-symmetry relation} {\bf (A)}:

\

\noindent
{\bf (A)}: $(b_1, \ldots, b_n)=-(-b_1, \ldots, b_n)$, for all generating symbols $\beta$.

\

For clarity, we distinguish symbols in $\cM_n(G)$ and $\cM_n^-(G)$ with the following notation:
\begin{itemize}
    \item $\langle b_1,\ldots,b_n\rangle\in\cM_n(G)$,    \item $\langle b_1,\ldots,b_n\rangle^-\in\cM^-_n(G)$.
\end{itemize}

\begin{rema}
    The original definition of  relation {\bf (M)} in \cite{KPT} is more involved, but is equivalent to the version here, by \cite[Proposition 2.1]{HKTsmall}.  
\end{rema}

When $n=1$, we have
$$
\cM_1(G)=\begin{cases}
 \bZ^{\phi(N)}& G=\bZ/N, N\geq 1,\\
 0&\text{otherwise,}
\end{cases}
$$where $\phi(n)$ is Euler's totient function. 

When $n=2$, $\cM_2(G)$ can be nontrivial for cyclic and bi-cyclic groups. Below, we present results of numerical computations of $\bQ$-ranks of $\cM_2(G)$ and $\cM_2(G)^-$. Let 
$$
\cM_2(G)_\bQ:=\cM_2(G)\otimes\bQ,\quad\text{and}\quad\cM^-_2(G)_\bQ:=\cM^-_2(G)\otimes\bQ.
$$
In the following tables, $d$ and $d^{-}$ denote respectively
$$
\dim_\bQ(\cM_2(G)_\bQ)\quad\text{and}\quad\dim_\bQ(\cM_2^-(G)_\bQ).
$$

When $G=C_N$ is cyclic, we have 
$$
\begin{tabular}{c|ccccccccccccccccc}
${\it N}$           & 2 & 3 & 4 & 5   & 7   & 9  & 11 &12& 13 &16& 17 & 19            & 23& 29&31&37 \\
\hline 
$d$            & 0 & 1& 1 & 2   & 3  & 5   &  6 &7  & 8 &10& 13& 16  &23  &  36&41&58  \\

$d^-$         & 0 & 0 & 0  & 0  & 0 &   1  & 1  &  2 & 2 &  3    & 5 &    7     &12&  22   & 16&40    \\

\end{tabular}
$$

\

When $G=C_{N_1}\times C_{N_2}$ is bi-cyclic, we have 
$$
\begin{tabular}{c|ccccccccccccccc}
${\it N}_1$           & 2 & 2 & 2 & 2  & 2 & 2  & 3 & 3& 3   & 3 & 4 & 4 & 4            & 5 & 6 \\
\hline
${\it N}_2$           & 2 & 4 & 6 & 8  & 10 & 16 & 6 & 3 & 9  & 27 & 8 & 16 & 32      & 25 & 36  \\
\hline
$d$            & 0 & 2 & 3 & 6  & 7 & 21 &  15 & 7   & 37 &  235   & 33 & 105 &  353     &702  &  577  \\
$d^-$         & 0 & 0 & 0  & 1  & 1 &   9  & 7  &  3 & 19 &  163    & 17 &    65     &257&  502   & 433    \\

\end{tabular}\\
$$

\

In particular, when $G=C_{p}\times C_{p}$, for prime $p$, we have
 $$
 \begin{tabular}{c|ccccccccccccc}
${\it p}$           & 5 & 7    & 11 & 13& 17   & 19 & 23 & 29 & 31  & 37  \\
\hline
$d$              & 46 & 159    & 855 & 1602 &  4424    & 6759 & 14047 & 34314 &   44415 &  88254\\
$d^-$              & 22 & 87    & 555 & 1098 & 3272   & 5139 & 11143  & 28434 &   37215 &  75942\\
\end{tabular}\\
$$
It was discovered and proved in \cite{KPT} that
$$
\dim(\cM_2^-(C_N)_\bQ)=
\begin{cases}
1-\frac{\phi(N)+\phi(N/2)}{2}+\frac{N\cdot\phi(N)}{24}\displaystyle\cdot{\prod_{p\mid N}}(1+\frac1p)&\text{$N$ even,}\\1-\frac{\phi(N)}{2}+\frac{N\cdot\phi(N)}{24}\cdot\displaystyle{\prod_{p\mid N}}(1+\frac1p)&\text{$N$ odd.}
\end{cases}
$$ 
The proof is based on an isomorphism between $\cM_2^-(C_N)_\bQ$ and the space of modular symbols of the congruence subgroups $\Gamma_1(N)$. From the tables above, we speculate the following identities
$$
\dim(\cM_2(C_p\times C_p)_\bQ )\stackrel{?}{=} \frac{(p-1)(p^3+6p^2-p+6)}{24},
$$
$$
\dim(\cM^-_2(C_p\times C_p)_\bQ )\stackrel{?}{=}  \frac{(p-1)(p^3-p+12)}{24},
$$
also signaling a strong connection to modular forms. The remaining part of this paper is dedicated to a proof of these two identities in the general setting. 

First, observe that the common factor ${(p-1)}$ indicates that the structure of $\cM_2(G)$ and $\cM_2^-(G)$ can be simplified when $G$ is a bi-cyclic group. We explain in detail the simplification for $\cM_2^-(G)$ below. The same argument also applies to $\cM_2(G)$.

\subsection*{Bi-cyclic groups} Let $G=\C$ be a finite bi-cyclic group. By definition, the $\bZ$-module $\cM_2^-(G)$ is generated by symbols 
$$
\beta:=\symb{(a_1, b_1), (a_2, b_2)}^-
$$
such that $$a_1, a_2\in C_N,\quad b_1, b_2\in C_{MN},\quad\bZ(a_1, b_1)+\bZ(a_2, b_2)=\C,$$ 
and subject to  relations
\begin{itemize}
    \item 
$\beta=\symb{(a_2, b_2), (a_1, b_1)}^-$,
    \item 
    $\beta=\symb{(a_1-a_2, b_1-b_2), (a_2, b_2)}^-+\symb{(a_1, b_1), (a_2-a_1, b_2-b_1)}^-,$
    \item 
    $\beta=-\symb{(-a_1, -b_1), (a_2, b_2)}^-.$
\end{itemize}
Formally, we can also denote $\beta$ by a $2\times 2$ matrix 
$$
\begin{pmatrix}a_1&a_2\\b_1&b_2\end{pmatrix}
$$ 
and assign a determinant:
$$
\det(\beta):=a_1b_2-a_2b_1\in(\bZ/N)^\times,
$$ 
where the operation takes place modulo $N$. From the defining relations {\bf(O)}, {\bf (M)} and {\bf (A)}, one can see that the linear combinations of symbols with the same determinant up to $\pm1$ form a submodule of $\cM_2^-(G).$ More precisely, for $k\in(\bZ/N)^\times$, let
\begin{align}\label{eqn:S2K}
  \mathcal{S}_{2,k}(G)
\end{align}
be the {\em finite set} consisting of matrices/symbols
$$
\beta:=\begin{pmatrix}
    a_1&a_2\\
    b_1&b_2
\end{pmatrix}=\langle(a_1,b_1),(a_2,b_2)\rangle^-
$$ 
such that 
\begin{itemize}
    \item  $(a_1,b_1),(a_2,b_2)\in (\C)^\vee$,
    \item $\bZ (a_1,b_1)+\bZ (a_2,b_2)=(\C)^\vee$,
    \item $\det(\beta)=k\pmod N$,
\end{itemize}
and 
$$
\cM_{2,k}^-(G)
$$ 
be the $\bZ$-module freely generated by elements in the set
$$
\cS_{2,k}(G)\cup\cS_{2,-k}(G)
$$
subject to relations {\bf(O)}, {\bf(M)} and {\bf(A)}. It follows that $\cM_{2,k}^-(G)$ can be naturally identified as a submodule of $\cM_2^-(G)$. Moreover, the algebraic structure of $\cM_{2,k}^-(G)$ is independent of $k$: consider the maps
$$
\cM_{2,1}^-(G)\to\cM_{2,k}^-(G),\quad\symb{(a_1,b_1),(a_2,b_2)}^-\mapsto\symb{(ka_1, b_1),(ka_2, b_2)}^-\,\,;
$$
$$
\cM_{2,k}^-(G)\to\cM_{2,1}^-(G),\quad\symb{(a_1,b_1),(a_2,b_2)}^-\mapsto\symb{(a_1/k, b_1),(a_2/k, b_2)}^-.
$$
These maps respect the defining relations and are inverse to each other. It follows that we have isomorphisms of $\bZ$-modules, when $N\geq 3$:
$$
\cM_2^-(G)\simeq\bigoplus_{k\in (\bZ/N)^\times/\langle\pm1\rangle} \cM^-_{2,k}(G)\simeq\bigoplus_{\frac{\phi(N)}{2}\,\mathrm{ copies}}\cM^-_{2,1}(G).
$$
\subsection*{Multiplication and Co-multiplication}
Given an exact sequence of finite abelian groups 
$$
0\to G'\to G\to G''\to0,
$$ 
consider the dual sequence of their character groups
$$
0\to A''\to A\to A'\to0.
$$
For all integers $n=n'+n'', n',n''\geq1$, one can define a $\bZ$-bilinear {\em multiplication} map 
$$
\nabla: \cM_{n'}(G')\otimes\cM_{n''}(G'')\to\cM_{n}(G)
$$ 
given on the generators by $$\symb{a_1',\ldots, a_{n'}'}\otimes\symb{a_1'',\ldots, a_{n''}''}\to\sum\symb{a_1,\ldots,a_{n'},a_1'',\ldots, a_{n''}''},
$$ where the sum is over all possible lifts $a_i\in A$ of $a_i'\in A'$; and $a_i''\in A$ are understood via the embedding $A''\hookrightarrow A$.

Dual to this construction is the $\bZ$-linear {\em co-multiplication} map when $G''$ is non-trivial:
\begin{align}
\label{eq:comultdef}
\Delta: \cM_{n}(G)\to\cM_{n'}(G')\otimes \cM_{n''}^-(G'').
\end{align}
This map is defined on the generators by 
$$
\symb{a_1,\cdots ,a_n}\mapsto\sum \symb{a_{I'}\mod A''}\otimes\symb{a_{I''}}^-,
$$
where the sum is over all partition of $\{1,\ldots,n\}=I'\cup I''$ such that
 \begin{itemize}
    \item $\#I'=n', \quad\#I''=n'';$
    \item for all $j\in I''$, $a_j\in A''\subset A$; and for any $i\in I'$, $a_i \mod A''$ is understood as projection of $a_i\in A$ in $A/A''$;
    \item the elements $a_j, j\in I''$, span $A''$.
\end{itemize} 
The correctness of $\nabla$ and $\Delta$ can be verified directly \cite{KPT}; they maps also descend to well-defined $\bZ$-module homomorphisms
$$
\nabla^-: \cM_{n'}^-(G')\otimes\cM_{n''}^-(G'')\to\cM_{n}^-(G),
$$
$$
\Delta^-: \cM_{n}^-(G)\to\cM^-_{n'}(G')\otimes \cM_{n''}^-(G'').
$$
\section{Congruence subgroups and Modular Symbols}\label{sect:modsymb}
\subsection*{Congruence subgroups} Connections between $\cM_2^-(C_N)$ and a classical congruence subgroup 
$$
\Gamma_1(N)=
\left\{
\gamma\in\SL_2(\bZ): \gamma=
\begin{pmatrix}
    1&*\\
    0&1
\end{pmatrix}
\right\},\quad N\geq2,
$$
were discovered in \cite[Section 11]{KPT}. To extend their results to bi-cyclic groups, we introduce a new family of congruence subgroups
\begin{align}\label{eq:Gammadef1}
    \Gamma(N, MN):=\left\{\begin{pmatrix}a&b\\c&d\end{pmatrix}\in\SL_2(\bZ)\,\,\middle\vert\,\,\begin{array}{l} a\equiv 1\pmod{N}\\ b\equiv 0\pmod {N}\\ c\equiv 0\pmod{MN}\\ d\equiv 1\pmod {MN}
\end{array}\right\},\quad N\geq2.
\end{align}
To see that $\Gamma(N, MN)$ is indeed a congruence subgroup, one can check that the definition \eqref{eq:Gammadef1} forces
$$
a\equiv 1\mod{MN},
$$ 
leading to an equivalent description of $\Gamma(N,MN)$:
\begin{align}\label{eq:Gammadef2}
    \Gamma(N, MN)=\left\{\begin{pmatrix}a&b\\c&d\end{pmatrix}\in\SL_2(\bZ)\,\,\middle\vert\,\,\begin{array}{l}a\equiv 1\pmod{MN}\\ b\equiv 0\pmod {N}\\c\equiv 0\pmod{MN}\\ d\equiv 1\pmod {MN}
\end{array}\right\},\quad N\geq2.
\end{align}
Using \eqref{eq:Gammadef2}, one can easily verify the following inclusion relations
$$
\SL_2(\bZ)\supset\Gamma_1(MN)\supset\Gamma(N, MN)\supset
\Gamma(MN)
$$
and conclude that $\Gamma(N,MN)$ is a congruence subgroup.
\begin{lemm}\label{lemm:quogmM}
$[\Gamma(N,MN):\Gamma(MN)]=M$.
\end{lemm}

\begin{proof}

Consider the surjective group homomorphism:
$$
\Gamma(N, MN)\to\bZ/M\bZ,\quad
\begin{pmatrix}
a&b\\c&d
\end{pmatrix}
\mapsto \frac bN\pmod M.
$$
The kernel of the homomorphism is $\Gamma(MN)$. In particular, 
$$
\Gamma(N,MN)/\Gamma(MN)\simeq\bZ/m\bZ.
$$
\end{proof}
To study the space of Manin symbols associated with $\Gamma(N,MN)$, one needs a description of the right cosets $\Gamma(N, MN)\setminus\SL_2(\bZ)$. Now, we show that $\Gamma(N, MN)\setminus\SL_2(\bZ)$ coincides with the set $\cS_{2,1}(\C)$ introduced in \eqref{eqn:S2K}. Consider a natural map:
\begin{align}\label{eq:ltcsiso}
    \Gamma(N, M&N)\setminus\SL_2(\bZ)\to\cS_{2,1}(\C),\\
    \begin{pmatrix}
    a&b\\c&d
    \end{pmatrix}\mapsto &\begin{pmatrix}
    a\mod N&&b\mod N\\c\mod{MN}&&d\mod{MN}
    \end{pmatrix}.\nonumber
\end{align}
The correctness of \eqref{eq:ltcsiso} as a bijection between finite sets follows from elementary computations. Moreover, we have the following lemmas.
\begin{lemm}\label{lemm:equivclassgmn}
For $\gamma_i=\begin{pmatrix}
a_i&b_i\\c_i&d_i
\end{pmatrix}\in\SL_2(\bZ),\, i=1,2,$ one has 
$$
\begin{pmatrix}
a_1&b_1\\c_1&d_1
\end{pmatrix}\equiv\begin{pmatrix}
a_2&b_2\\c_2&d_2
\end{pmatrix}\pmod{\Gamma(N,MN)}
$$
if and only if 
$
\quad\begin{cases}
     a_1\equiv a_2\pmod N, & c_1\equiv c_2\pmod {MN},  \\
     b_1\equiv b_2\pmod {N}, &  d_1\equiv d_2\pmod{MN}.
\end{cases}
$
\end{lemm}

\begin{proof}
Basic modular arithmetic, as in \cite[Lemma 3.1]{cremonagamma1}. 
\end{proof}

\begin{lemm}\label{lemm:mapgnrtcond}
Let $\begin{pmatrix}
a&b\\c&d
\end{pmatrix}\in\SL_2(\bZ)$, and $a', b', c', d'\in \bZ$  such that
$$
\begin{cases}
     a'\equiv a\pmod N, & c'\equiv c\pmod {MN},  \\
     b'\equiv b\pmod {N}, &  d'\equiv d\pmod{MN},
\end{cases}
$$
with $0\leq a',b'< N$ and $0\leq c',d'< MN$. Then we have  
$$
\begin{pmatrix}
a'&b'\\c'&d'
\end{pmatrix}\in\cS_{2,1}(\C).
$$
\end{lemm}

\begin{proof}
It suffices to check $\bZ(a',c')+\bZ(b', d')=\C$. Indeed,
$$
\begin{pmatrix}
a'&b'\\c'&d'
\end{pmatrix}\begin{pmatrix}     
d&-b\\-c&a
\end{pmatrix}=\begin{pmatrix}
a'd-b'c&-a'b+ab'\\c'd-d'c&-c'b+ad'
\end{pmatrix}\in\Gamma(N, MN),
$$
since $ad-bc=1.$ This shows $(a', c')$ and $(b', d')$ generate the generators $(0,1)$ and $(1,0)\in\C$. 
\end{proof}
\begin{prop}
The map \eqref{eq:ltcsiso} is a well-defined bijection between finite sets. 
\end{prop}
\begin{proof}
Lemmas \ref{lemm:equivclassgmn} and \ref{lemm:mapgnrtcond} implies \eqref{eq:ltcsiso} is a well-defined injection. It suffices to show it is also surjective. Let 
 $$
 \beta=\begin{pmatrix}
a&b\\c&d
\end{pmatrix}\in\cS_{2,1}(\C).
$$ 
By definition, one has $ad-bc=1+l_1N$ for some $l_1$. The generating condition implies that $\gcd(c,d, M)=1$. So there exists $k_1, k_2\in C_M$ such that
$$
k_1d-k_2c=-l_1\pmod M.
$$
Put
$$
\gamma=\begin{pmatrix}
    a+k_1N&b+k_2N\\c&d
    \end{pmatrix},
    $$
One computes that
$
    \det(\gamma)\equiv1\pmod {MN}
$, i.e., $\gamma\in\SL_2(\bZ/MN)$.
 Let $\overline\gamma$ be a lift of $\gamma$ in $\SL_2(\bZ)$ under the surjection $\SL_2(\bZ)\to\SL_2(\bZ/MN).$ The lift $\overline\gamma$ is mapped to $\beta$ under the map \eqref{eq:ltcsiso}, proving surjectivity.
\end{proof}

\
\subsection*{Modular symbols}
We follow Manin's definition of modular symbols \cite[Section 1.7]{maninmodsymb}. Given the bijection \eqref{eq:ltcsiso}, the space $\bM_2(\Gamma(N,MN))$ of modular symbols of weight $2$ for $\Gamma(N,MN)$ is defined via generators 
$$
\begin{pmatrix}
    a&b\\
    c&d
\end{pmatrix}\in\cS_{2,1}(\C)
$$
subject to relations
\begin{enumerate}
    \item[(1)] $\begin{pmatrix} a&b\\c&d
\end{pmatrix}+\begin{pmatrix} b&-a\\d&-c
\end{pmatrix}=0,$

\item[(2)]
$\begin{pmatrix} a&b\\c&d
\end{pmatrix} +\begin{pmatrix} a+b&-a\\c+d&-c
\end{pmatrix}+\begin{pmatrix} b&-a-b\\d&-c-d
\end{pmatrix}=0,$

\item[(3)]
$\begin{pmatrix} a&b\\c&d
\end{pmatrix}=0\,\,$ if $\begin{pmatrix} a&b\\c&d
\end{pmatrix} =\begin{pmatrix} b&-a\\d&-c
\end{pmatrix}$ or $\begin{pmatrix} a+b&-a\\c+d&-c
\end{pmatrix}.$
\end{enumerate}
Relation (3) guarantees that the space of modular symbols is torsion-free. But for $\Gamma(N,MN)$,  relation (3) is redundant as the condition in (3) is never satisfied. Using relation $(1)$, relation $(2)$ can be rewritten:
\begin{align*}
    0&\stackrel{(2)}{=}\begin{pmatrix} b&-a\\d&-c
\end{pmatrix} +\begin{pmatrix} b-a&-b\\d-c&-d
\end{pmatrix}+\begin{pmatrix} -a&a-b\\-c&c-d
\end{pmatrix}\\
&\stackrel{(1)}{=}-\begin{pmatrix} a&b\\c&d
\end{pmatrix} +\begin{pmatrix} a-b&b\\c-d&d
\end{pmatrix}+\begin{pmatrix} a&b-a\\c&d-c
\end{pmatrix}.
\end{align*}
Equivalently, one can rewrite defining relations of $\bM_2(\Gamma(N,MN))$ as

\begin{enumerate}
    \item[\text{{\bf(R1)}}] $\begin{pmatrix} a&b\\c&d
\end{pmatrix} =-\begin{pmatrix} b&-a\\d&-c
\end{pmatrix},$
\item[\text{{\bf(R2)}}]
$\begin{pmatrix} a&b\\c&d
\end{pmatrix} =\begin{pmatrix} a-b&b\\c-d&d
\end{pmatrix}+\begin{pmatrix} a&b-a\\c&d-c
\end{pmatrix}.$\\
\end{enumerate}
\begin{prop}\label{prop:t1modiso}
The $\bZ$-modules $\cM_{2,1}^-(\C)$ and $\bM_2(\Gamma(N, MN))$ are isomorphic when $N\in\bZ_{>2}$ and $M\in\bZ_{\geq1}$.
\end{prop}
\begin{proof}
When $N>2$, consider the map
\begin{align}\label{eq: t1modsymbmap} 
\cM_{2,1}^-(\C)\to\bM_2(\Gamma(N, MN)),
\end{align}
$$
\symb{(a_1,b_1),(a_2,b_2)}^-\mapsto
\begin{cases}
    \begin{pmatrix}
a_1& a_2\\b_1&b_2
\end{pmatrix}&\text{if } a_1b_2-a_2b_1=1\pmod N,\\[0.5cm]
\begin{pmatrix}
a_2& a_1\\b_2&b_1
\end{pmatrix}&\text{if } a_1b_2-a_2b_1=-1\pmod N.
\end{cases}
$$
The correctness of the map $\eqref{eq: t1modsymbmap}$ can be verified directly:
\begin{itemize}
    \item It is compatible with the relation {\bf(O)} by construction.
    \item Relation {\bf(M)} is identical to relation {\bf (R2)} and preserves the determinants of the symbols.
     \item It is compatible with relation {\bf(A)} due to the defining relation {\bf (R1)} of $\bM_2(\Gamma(N, MN))$.
\end{itemize}
Similarly, one can check that the map given by 
$$
    \bM_2(\Gamma(N, MN))\to\cM_{2,1}^-(\C),\quad \begin{pmatrix}
a&b\\c&d
\end{pmatrix}\mapsto\symb{(a,c),(b,d)}^-
$$
is a well-defined inverse homomorphism to \eqref{eq: t1modsymbmap}.
\end{proof}

When $N=2$, the map $\eqref{eq: t1modsymbmap}$ in the proof above is not well-defined as $\pm1$ are not distinguishable modulo 2. But in this case, the generating sets of $\cM_2^-(C_2\times C_{2M})$ and $\bM_2(\Gamma(2,2M))$ coincide: 
$
\cS_2(C_2\times C_{2M})$ is simply the free $\bZ$-module generated by elements in $\cS_{2,1}(C_2\times C_{2M}).
$ 
We can then consider the $\bZ$-module
$$
\bM^-_2(\Gamma(2,2M))
$$
defined as the quotient of $\cS_2(C_2\times C_{2M})$ by relations {\bf(R1)} and {\bf(R2)}, i.e., the quotient of $\bM_2(\Gamma(2,2M))$ by
$$
{\bf (O)}:
\begin{pmatrix}
a&b\\c&d
\end{pmatrix}=\begin{pmatrix}
b&a\\d&c
\end{pmatrix}.
$$
\begin{prop}\label{prop:n=2case}
The $\bZ$-modules $\cM_2^-(C_2\times C_{2M})$ and $\bM_2^-(\Gamma(2,2M))$ are isomorphic  for all integers $M\in\bZ_{\geq1}$.
\end{prop}
\begin{proof}
With the presence of {\bf (O)}, the relation {\bf(R1)} is identical to {\bf (A)}. It follows that relations {\bf(R1)} and {\bf (R2)} generate the same submodule of $\cS_2(C_2\times C_{2M})$ as {\bf (M)} and {\bf (A)} does.
\end{proof}

It is classically known that $\bM_2(\Gamma(N, MN))$ can be identified as 
$$
H_1(\overline{X(N,MN)},\bZ),
$$
the first homology group of the complex modular curve $X(N, MN)$ compactified with respect to the cusps \cite[Theorem 1.9]{maninmodsymb}. We follow definitions in \cite[Chapter 1.3]{shimuraautomorphic}:  
\begin{itemize}
    \item $X(N, MN):=\Gamma(N, MN)\backslash \mh,$ where $\mh$ is the upper half-plane,
     \item $\bP^1(\bQ):=\bQ\cup\{\infty\}$, cusps are the elements of $\bP^1(\bQ)/\Gamma(N, MN)$,
     \item $\mh^*:=\mh\cup\bP^1(\bQ)$ is the extended upper half-plane,
        \item $\overline{X(N, MN)}:=\Gamma(N, MN)\backslash 
    \mh^*$.
   
\end{itemize}
In particular, a symbol $
\begin{pmatrix}
    a&b\\
    c&d
\end{pmatrix}$
corresponds to the image in $X(N,MN)$ of the geodesic path from $a/c$ to $b/d$, where $a,b,c$ and $d$ are naturally considered as integers. Moreover, $\bM^-_2(\Gamma(2,2M))$ can be identified as the $(-1)$-eigenspace of the antiholomorphic involution on $X(2, 2M)$ given by the map $\tau\mapsto-\bar\tau, \tau\in\cH$, on the universal cover. On modular symbols, $\iota$ takes the form 
$$
\iota:\begin{pmatrix}
a&b\\c&d
\end{pmatrix}\mapsto\begin{pmatrix}
a&-b\\-c&d
\end{pmatrix}\stackrel{\text{{\bf(R1)}}}{=}-\begin{pmatrix}
-b&-a\\d&c
\end{pmatrix}\stackrel{\mod2}{=}-\begin{pmatrix}
b&a\\d&c
\end{pmatrix}.
$$
This forces a $2$-torsion in $\bM_2^-(\Gamma(2,2M))$ each time a cusp different from $\infty$ is fixed by $\iota$.

Concretely, these imply that
\begin{align}\label{eqn:dimform}
\dim(\bM_2(\Gamma(N, MN)_\bQ)=2g(N, MN)+\varepsilon_\infty(N, MN)-1,\\
\dim(\bM_2^-(\Gamma(2, 2M)_\bQ)=g(2, 2M)+\frac{\varepsilon_\infty(2, 2M)-\varepsilon(2, 2M)}{2},\nonumber\\
\Tor(\bM_2(\Gamma(N, MN))=0,\,\,
\Tor(\bM_2^-(\Gamma(2, 2M)))=(\bZ/2)^{\varepsilon(2,2M)-1},\nonumber
\end{align}
where 
\begin{itemize}
    \item $g(N, MN)$ is the genus of $\overline{X(N, MN)}$ as a compact Riemann surface,
    
    \item$\varepsilon_\infty(N, MN)$ is the number of cusps, i.e., the cardinality of $\bP^1(\bQ)/\Gamma(N, MN)$.
    
    \item$\varepsilon(2, 2M)$ is the number of cusps fixed by the anti-holomorphic involution on $X(2, 2M)$.
    \item $\Tor$ refers to the torsion subgroup.
\end{itemize}
We compute each term appearing in \eqref{eqn:dimform}. It is well-known that 
$$
|\bP^1(\bQ)/\Gamma(MN)|=\frac{M^2N^2}{2}\cdot\displaystyle{\prod_{\substack{p\vert MN}}(1-p^{-2})}.
$$
Recall from  Lemma \ref{lemm:quogmM} that $[\Gamma(N, MN):\Gamma(MN)]=M$. Then 
$$
\varepsilon_\infty(N,MN)=\frac{MN^2}{2}\cdot\displaystyle{\prod_{\substack{p\vert MN}}(1-p^{-2})}.
$$
Using the genus formula of modular curves \cite[Theorem 3.1.1]{diamondshurman}, we obtain for $N\geq3$ and $M\geq1$:
$$
g(N, MN)=1+\frac{MN^2(MN-6)}{24}\cdot\displaystyle{\prod_{\substack{p\mid MN}}(1-p^{-2})}.
$$
To compute $\varepsilon(2,2M)$, first observe that
$$
\Gamma(2,2M)=\bigcup_{j\in\bZ/M}\Gamma(2M)\cdot\begin{pmatrix}
    1&2j\\
    0&1
\end{pmatrix}.
$$
Two reduced rational numbers $a/c$ and $a'/c'$ lie in the same equivalence class of cusps in $\bP^1(\bQ)/\Gamma(2,2M)$ if and only if  
$$
\frac ac\equiv\frac{a'}{c'}+2j\pmod{\Gamma(2M)} \quad\text{for some } j\in\bZ/M,
$$
if and only if \cite[Proposition 3.8.3]{diamondshurman}
$$
(a', c')\equiv\pm(a+2jc,c) \pmod{2M},\quad\text{for some}\,\, j\in\bZ/M.
$$
A counting argument leads to 
$$
    \varepsilon(2,2M)=2\phi(M)+\phi(2M),\quad M>2.
$$
We summarize the computations above and results in \cite[Section 11]{KPT}:

\begin{prop}\label{prop:mainformminus} Let $G$ be a finite abelian group. Then 
\begin{itemize}
    \item When $G=C_N$, $N\geq 5$ and $N$ is even, 
    $$
    \dim(\cM_2^-(G)_\bQ)=1-\frac{\phi(N)+\phi(N/2)}{2}+\frac{N\cdot\phi(N)}{24}\displaystyle\cdot{\prod_{p\mid N}}(1+\frac1p),
    $$
    $$
    \Tor(\cM_2^-(G))=(\bZ/2)^{\phi(N)+\phi(N/2)-1}.
    $$
    \item When $G=C_N$, $N\geq 5$ and $N$ is odd,
    $$
    \dim(\cM_2^-(G)_\bQ)=1-\frac{\phi(N)}{2}+\frac{N\cdot\phi(N)}{24}\cdot\displaystyle{\prod_{p\mid N}}(1+\frac1p),
    $$
    $$
    \Tor(\cM_2^-(G))=(\bZ/2)^{\phi(N)-1}.
    $$
\item When $G=C_2\times C_{2M}$, $M\geq3$,
$$\dim(\cM^-_2(G)_\bQ)=1-\phi(M)-\frac{\phi(2M)}{2}+\frac{M^2}{3}\cdot\displaystyle \prod_{p\vert MN} (1-p^{-2}),
$$
$$
\Tor(\cM^-_2(G))=(\bZ/2)^{2\phi(M)+\phi(2M)-1}.
$$
\item When $G=\C$, $N\geq 3$, $M\geq 1$,
$$
\dim(\cM_2^-(G)_\bQ)=\frac{\phi(N)}{2}\left(1+\frac{M^2N^3}{12}\cdot\displaystyle \prod_{p\vert MN} (1-p^{-2})\right),
$$
$$
\Tor(\cM^-_2(G))=0.
$$ 
\item 
$
\cM^-_2(C_2)=\cM^-_2(C_3)=\bZ/2, \quad\cM^-_2(C_4)=\cM^-_2(C_2^2)=(\bZ/2)^2.$
\item $\cM_2^-(G)=0$ if $G$ is not in any of the cases above.
\end{itemize} 
\end{prop}

\section{Dimensional Formulae}\label{sect:formulae}
Consider the natural quotient map of $\cM_2(G)$ by relation {\bf(A)}
$$
\mu^-:\cM_2(G)\to\cM_2^-(G).
$$
In this section, we determine the $\bQ$-rank of the kernel of $\mu^-$. First, we introduce an auxiliary group
$$
\cM^+_1(G)
$$
defined as the quotient of $\mathcal{M}_1(G)=\mathcal{S}_1(G)$ by the relation
$$
{\bf(P)}: \langle a_1 \rangle=\langle -a_1\rangle,
$$
and denote by $\symb{a_1}^+\in\mathcal{M}_1^+(G)$ the image of $\symb{a}\in\mathcal{M}_1(G)$ under the natural projection 
$$
\mu^+ : \mathcal{M}_1(G)\to\mathcal{M}_1^+(G).
$$
We have
$$
\mathcal{M}_1^+(G)=
\begin{cases}
\bZ^{\frac{\phi(N)}{2}}&G=C_N, N>2,\\ 
\bZ&G=C_N, N=1,2,\\
0&otherwise.
\end{cases}
$$
Given a finite abelian group $G$ and a subgroup $G'\subsetneq G$ such that $G'=C_d$ for some $d\in\bZ_{\geq1}$, there is a map
\begin{align}
\label{eq:maincomult}& \nu_{G'}: \mathcal{M}_n(G)\ra \mathcal{M}_1^+(G')\otimes\mathcal{M}_{n-1}^-(G''),
\end{align}
obatined as the composition of the co-multiplication map and $\mu^+$. Notice that $\nu_{G'}$ is non-trivial only when $G'$ is cyclic. Put
$$
\nu:=\bigoplus_{G'\subsetneq G}\nu_{G'}, 
$$
where the sum runs through all proper cyclic subgroups (including the trivial one) $G'\subsetneq G$. We will show that the restriction of $\nu$ to  
$$
\cK_n(G):=\ker\left(\mathcal{M}_n(G)\to\mathcal{M}_n^-(G)\right)
$$
is an isomorphism over $\bQ$. Formally, consider the map
\begin{align}\label{eq:nu}
\nu_{\cK_n(G)}: \cK_n(G)\to \bigoplus_{\substack{G'\subsetneq G}}\mathcal{M}_1^+(G')\otimes\mathcal{M}_{n-1}^-(G/G').
\end{align}
We construct an inverse of $\nu_{\cK_n(G)}$ over $\bQ$:
\begin{align}\label{eq:psi}
\psi: \bigoplus_{\substack{G'\subsetneq G}}\mathcal{M}_1^+(G')\otimes\mathcal{M}_{n-1}^-(G/G')\to\cK_n(G)\end{align}
in the following way:

 \
 
Let $G'=C_{d}\subsetneq G$ be a cyclic subgroup of $G$. We denote by
$$
A, A',\text{and } A''
$$ 
the character group of 
$$
G, G',\text{and } G/G'
$$ 
respectively. For any 
$$
\symb{a}^+\in\mathcal{M}_1^+(C_{d_i})
$$ 
and  
$$
\symb{b_1, b_2, \ldots, b_{n-1}}^-\in\mathcal{M}_{n-1}^-(G/G'),
$$ 
we set
$$
\boldsymbol b:=\{b_1, b_2, \ldots, b_{n-1}\},
$$
and
$$
\boldsymbol\omega(a,\boldsymbol b):=\symb{a}^+\otimes\symb{b_1,\ldots, b_{n-1}}^-\in\mathcal{M}_1^+(G')\otimes\mathcal{M}_{n-1}^-(G/G').
$$
Find an arbitrary lift $a'\in A$ of $a\in A'$ and put 
$$
\boldsymbol\gamma(a,\boldsymbol b):=\symb{a', b_1,\ldots, b_{n-1}}+\symb{-a', b_1,\ldots, b_{n-1}}\in\cK_n(G),
$$ 
where $b_i$ are understood via the embedding $A''\subset A$.
Then we define
\begin{align}\label{eq:psidef}\psi(\boldsymbol\omega(a, \boldsymbol b)):=\frac12\,\boldsymbol\gamma(a,\boldsymbol  b).\end{align}
Notice that $\psi$ is defined over $\bQ$. It is not hard to see that 
\begin{align}
\label{eq:nudprop}
\nu_{G'}(\,\frac12\,\boldsymbol\gamma(a, \boldsymbol b)\,)=\boldsymbol\omega(a, \boldsymbol b)
\end{align}
and the map $\psi$ is compatible with relations {\bf(O)} and {\bf(M)}. It remains to check that the construction is independent of the lift $a'$ and $\psi$ is also compatible with relations {\bf(P)} and {\bf(A)} as a homomorphism between $\bQ$-vector spaces.
\begin{lemm}\label{lemm:psiwelldef}
With the notation above, the definition of $\psi$ is independent of the choice of the lift $a'$ of $a$.
\end{lemm}
\begin{proof}
    Let $a_1, a_2\in A$ be two lifts of $a\in A'$, i.e., there exists $g\in A''$ such that $a_2=a_1+g.$ Relations {\bf(S)} and {\bf(M)} imply that
\begin{align*}
\symb{a_1,b_1,\ldots}=\symb{a_1-b_1, b_1,\ldots}+\symb{a_1, b_1-a_1,\ldots},\\
\symb{b_1-a_1, b_1,\ldots}=\symb{-a_1, b_1,\ldots}+\symb{a_1, b_1-a_1,\ldots}.
\end{align*}
Taking the difference between the two lines above, one has 
$$
\symb{a_1,b_1,\ldots}+\symb{-a_1, b_1,\ldots}=\symb{a_1-b_1, b_1,\ldots}+\symb{b_1-a_1, b_1,\ldots}.
$$
Iterating this process with $b_i$, we obtain
$$
\symb{a_1,b_1,\ldots}+\symb{-a_1, b_1,\ldots}=\symb{a_1-\sum_{i=1}^{n-1}m_ib_i, b_1,\ldots}+\symb{\sum_{i=1}^{n-1}m_ib_i-a_1, b_1,\ldots}
$$
where $m_i\in\bZ_{\geq 0}$ for all $i$. Since $b_i$ generate $A''$, we conclude that
$$
\symb{a_1,b_1,\ldots}+\symb{-a_1, b_1,\ldots}=\symb{a_2, b_1,\ldots}+\symb{-a_2, b_1,\ldots}.
$$
\end{proof}
\noindent Notice that Lemma~\ref{lemm:psiwelldef} also implies that $\psi$ is compatible with the relation {\bf(P)}. Indeed, let $a'$ be a lift of $a\in A'$ in $A$ and $a''$ a lift of $-a\in A'$ in $A$. Then $a''=-a'+g$ for some $g\in A''$ and thus $\boldsymbol\gamma(a,\boldsymbol b)=\boldsymbol\gamma(-a,\boldsymbol b)$. The compatibility of $\psi$ with the relation {\bf(A)} is reduced to the following lemma.

\begin{lemm}\label{lemm:orbstab}
Let  $n\geq2$ be an integer, $G$ be a finite abelian group and $\symb{a_1,\ldots,a_n}$ be any generating symbol of $\cM_n(G)$, one has $$\sum_{\varepsilon_1,\varepsilon_2=\pm1}\symb{\varepsilon_1a_1,\varepsilon_2a_2, a_3,\ldots,a_n}=0\in\cM_n(G)\otimes \bQ.$$
\end{lemm}
\begin{proof}
For simplicity, we denote the sum in the assertion by
$$
\delta(\symb{a_1,\ldots,a_n}):=\sum_{\varepsilon_1,\varepsilon_2=\pm1}\symb{\varepsilon_1a_1,\varepsilon_2a_2, a_3,\ldots,a_n}.
$$ 
Consider a group action of $\SL_2(\bZ)$ on $\delta(\symb{a_1,\ldots,a_n})$ via
$$
\begin{pmatrix}a&b\\c&d\end{pmatrix}\cdot\delta(\symb{a_1, a_2, a_3,\ldots,a_n})=\delta(\symb{aa_1+ba_2, ca_1+da_2, a_3,\ldots,a_n}).
$$
Equivalently, we can view this as an action of $\SL_2(\bZ)$ on $(G^\vee)^2$. The action is in fact trivial in $\cM_n(G)$. It suffices to check this on generators of $\SL_2(\bZ)$: 
$$
\begin{pmatrix}0&1\\-1&0\end{pmatrix} \qquad\text{and}\qquad\begin{pmatrix}1&1\\0&1\end{pmatrix}.
$$ 
By symmetry, it is clear that 
$$
\delta(\symb{a_1, a_2,\ldots, a_n})=\delta(\symb{a_2, -a_1,\ldots,a_n}).
$$
On the other hand, one has
\begin{align*}
&\phantom{ ==}\delta(\symb{a_1+a_2, a_1, a_3, \ldots, a_n})\\
&=\symb{a_1+a_2,a_1,\ldots}+\symb{-a_1-a_2,-a_1,\ldots}+\symb{a_1+a_2,-a_1,\ldots}+\\
&\phantom{ ==}\symb{-a_1-a_2,a_1,\ldots}\\
&\phantom{ ==}\text{applying {\bf(M)} to the first two terms above}\\
&=\symb{a_1,a_2,\ldots}+\symb{-a_1,-a_2,\ldots}+\symb{a_1+a_2,-a_2,\ldots}+\\
&\phantom{ ==}\symb{-a_1-a_2,a_1,\ldots}+\symb{-a_1-a_2, a_2,\ldots}+\symb{a_1+a_2,-a_1,\ldots}\\
&\phantom{ ==}\text{applying ($\mathbf M$) to the last four terms above}\\
&=\symb{a_1,a_2,\ldots}+\symb{-a_1,-a_2,\ldots}+\symb{a_1,-a_2,\ldots}+\symb{-a_1,a_2,\ldots}\\
&=\delta (\symb{a_1,a_2,\ldots,a_n}).
\end{align*}
Consider 
\begin{align}
\label{eq:firstavg}
    S:=\sum_{a,b}\symb{a,b, a_3, \ldots, a_n},
\end{align}
where the sum runs over the $\SL_2(\bZ)$-orbit of $(a_1,a_2)$ in $(G^\vee)^2$. Observe that the orbit is finite as $G$ is a finite group. Applying relation {\bf(M)} to each term in the sum, one finds that
\begin{align*}
S&=\sum_{a,b}\symb{a-b,b, a_3, \ldots, a_n}+\symb{a,b-a, a_3, \ldots, a_n}\\
&=2\sum_{a,b}\symb{a,b, a_3, \ldots, a_n}
\end{align*}
since
$$
\begin{pmatrix}
a-b\\b\end{pmatrix}=\begin{pmatrix}
1&-1\\0&1\end{pmatrix}\cdot\begin{pmatrix}
a\\b\end{pmatrix},\quad \begin{pmatrix}
a\\b-a\end{pmatrix}=\begin{pmatrix}
1&0\\-1&1\end{pmatrix}\cdot\begin{pmatrix}
a\\b\end{pmatrix}.
$$
Similarly, averaging $\delta$ over this orbit leads to 
\begin{align*}
&\phantom{=}\sum_{a,b}\delta(\symb{a,b,a_3,\ldots,a_n})\\
&=\sum_{a,b}\symb{a,b,\ldots}+\symb{-a,b,\ldots}+\symb{a,-b, \ldots}+\symb{-a,b,\ldots}\\
&\phantom{==}\text{applying \eqref{eq:firstavg} to each term}\\
&=2\cdot\sum_{a,b}\delta(\symb{a,b,a_3,\ldots,a_n}).
\end{align*}
Recall that $\delta$ is invariant under the $\SL_2(\bZ)$-action. We conclude that
$$
\delta(\symb{a_1,\ldots, a_n})=0\in\cM_n(G)\otimes\bQ.
$$
\end{proof}

\begin{prop}\label{prop:kerisomo}
The map $\psi$ is well-defined over $\bQ$. In addition, $\nu_{\cK_n(G)}$ and $\psi$ are inverse to each other over $\bQ$.
\end{prop}
\begin{proof}
The correctness of $\psi$ is due to Lemma~\ref{lemm:psiwelldef} and ~\ref{lemm:orbstab}. By definition, $\cK_n(G)$ is generated by 
$$
\boldsymbol\gamma(a, \boldsymbol b)=\symb{a,b_1,\ldots,b_{n-1}}+\symb{-a,b_1,\ldots,b_{n-1}}.
$$
Let $G'$ be the subgroup of $G$ such that 
$$
\sum_{i=1}^{n-1}\bZ b_i=(G/G')^\vee.
$$
The definition of the co-multiplication map ensures that
$$
\nu_{\cK_n(G)}(\boldsymbol\gamma(a,\boldsymbol b))=\nu_{G'}(\boldsymbol\gamma(a,\boldsymbol b))
$$
and one can deduce from \eqref{eq:nudprop} that
$$
\psi\circ\nu_{\cK_n(G)}(\boldsymbol\gamma(a,\boldsymbol b))=\psi(2\,\boldsymbol\omega(a,\boldsymbol b))=\boldsymbol\gamma(a,\boldsymbol b),
$$
where the last equality holds by Lemma~\ref{lemm:psiwelldef}. Similarly, for any
$$
\boldsymbol\omega(a,\boldsymbol b)=\symb{a}^+\otimes\symb{b_1,\ldots, b_{n-1}}^-\in\mathcal{M}_1^+(G')\otimes\mathcal{M}_{n-1}^-(G/G'),
$$
one has
$$
\nu_{\cK_n(G)}\circ\psi(\boldsymbol\omega(a,\boldsymbol b))=\nu_{\cK_n(G)}(\frac12\boldsymbol\gamma(a,\boldsymbol b))=\boldsymbol\omega(a,\boldsymbol b).
$$
It follows that $\psi$ and $\nu_{\cK_n(G)}$ are inverse to each other as homomorphisms between $\bQ$-vector spaces.
\end{proof}

\subsection*{Dimensional Formulae}Proposition~\ref{prop:kerisomo} provides an effective computation for 
$$
\dim(\cM_n(G)_\bQ)-\dim(\cM_n^-(G)_\bQ).
$$
In particular, it implies the hypothetical formula (note that the original formula in \cite[Section 11]{KPT} is wrong)
\begin{align*}
 &\phantom{===}\dim(\cM_2(C_N)_\bQ)-\dim(\cM_2^-(C_N)_\bQ)\\
 &\stackrel{N>5}{=}\begin{cases}
 \displaystyle{\frac{\phi(N)}{2}+\frac14\sum_{d\mid N,3\leq d\leq N/3}\phi(d)\phi(N/d)}& N \text{ odd},\\
  \displaystyle{\frac{\phi(N)+\phi(\frac N2)}{2}+\frac14\sum_{d\mid N,3\leq d\leq N/3}\phi(d)\phi(N/d)}& N \text{ even}.
 \end{cases}
\end{align*}
Combining this with Proposition~\ref{prop:mainformminus}, we obtain an effective computation for 
$$
\dim(\cM_2(G)_\bQ).
$$
For example, when $G=C_p\times C_p$, $p$ an odd prime, one has
$$
\dim(\cM_2(C_p\times C_p)\otimes\bQ)-\dim(\bQ\otimes\cM^-_2(C_p\times C_p))=\cfrac{(p+1)(p-1)^2}{4}
$$ 
and thus
$$
\dim(\cM_2(C_p\times C_p)\otimes\bQ)=\frac{(p-1)(p^3+6p^2-p+6)}{24},
$$ 
which is consistent with results of computer experiments recorded in Section~\ref{sect:back}.

\bibliographystyle{plain}
\bibliography{mod}
\end{document}